\documentclass[11pt,a4paper]{amsart}

\usepackage{amsthm, amsfonts, amssymb, amsmath, latexsym, enumerate,array}
\usepackage{amscd} 
\usepackage[all]{xy}
\usepackage{pstricks}
\usepackage{pgfplots}
\pgfplotsset{compat=1.15}
\usepackage{mathrsfs}
\usetikzlibrary{arrows}

\usepackage{stmaryrd}
\usepackage[utf8]{inputenc}
\usepackage{hyperref,cite,mdwlist,mathrsfs}
\usepackage{tikz}
\usepackage{pgf,tikz}
\usetikzlibrary{arrows}

\usepackage{anysize}
\marginsize{3.0cm}{3.0cm}{3.3cm}{3.3cm}
\usepackage{setspace}
\usepackage{soul}

\newtheorem{thm}{Theorem}[section]
\newtheorem{lem}[thm]{Lemma}

\newtheorem{prop}[thm]{Proposition}

\newtheorem*{thm*}{Theorem}
\newtheorem*{cnj*}{Conjecture}

\theoremstyle{definition}
\newtheorem{rmk}[thm]{Remark}
\newtheorem{eg}[thm]{Example}
\newtheorem{dfn}[thm]{Definition}

\newtheorem*{conj*}{Conjecture}

\newcommand{\sO}{\mathscr{O}}

\DeclareMathOperator{\rk}{rk}

\newcommand{\p}{\mathbb P}

\allowdisplaybreaks[1]
\title{ Saito criterion and its avatars }
\author[D. Faenzi]{Daniele Faenzi}
\address{Daniele Faenzi.
  Institut de Math{\'e}matiques de Bourgogne, UMR 5584 CNRS,
  Universit{\'e} de Bourgogne, 9 Avenue Alain
  Savary, BP 47870, 21078 Dijon Cedex, France}
\email{daniele.faenzi@u-bourgogne.fr}

\author[M. Jardim]{Marcos Jardim}
\address{Marcos Jardim. Universidade Estadual de Campinas (UNICAMP) \\ Instituto de Matemática, Estatística e Computação Científica (IMECC) \\ Departamento de Matem\'atica \\
Rua S\'ergio Buarque de Holanda, 651\\ 13083-859 Campinas-SP, Brazil}
\email{jardim@unicamp.br}

\author[J. Vallès]{Jean Vallès}
\address{Jean Vall\`es. Universit\'e de Pau et des Pays de l'Adour,
  LMAP-UMR CNRS 5142, 
  Avenue de l'Universit\'e - BP 1155 -
  64013 Pau Cedex, France}
 \email{jean.valles@univ-pau.fr}
\date{\today}

\thanks{
D.F. partially supported by FanoHK ANR-20-CE40-0023, SupToPhAG/EIPHI ANR-17-EURE-0002, Région Bourgogne-Franche-Comté. M.J. is supported by the CNPQ grant number 305601/2022-9 and the FAPESP Thematic Project number 2018/21391-1. All authors partially supported by Bridges ANR-21-CE40-0017 and CAPES/COFECUB project \emph{Moduli spaces in algebraic geometry and applications}, Capes reference number 88887.191919/2018-00, and ANR BRIDGES ANR-21-CE40-0017.
}



\begin{document}
\sloppy

\maketitle

\begin{center}
    \textit{We dedicate this paper to Enrique Arrondo on his 60th birthday.}
\end{center} 

\medskip

\abstractname{:}{  Saito in \cite{saito:logarithmic} gave a nice and efficient criterion to determine whether the module of logarithmic derivation associated with a reduced divisor in a complex variety is free or not. The aim of this note is to propose a new proof of this criterion, in the affine space, in the projective space, and for multiderivations, based on straightforward observations concerning free and reflexive modules. This point of view also allows us to prove a generalized version of the Saito criterion that applies to derivation modules associated with several polynomials introduced in \cite{faenzi-jardim-valles}.
}

\section{Introduction}
Let $R=k[x_1,\ldots, x_n]$ be a graded commutative ring over a field $k$ of characteristic zero. The partial derivatives $\frac{\partial}{\partial x_i}$ are denoted $\partial_{i}$. The $R$-module $\mathrm{Der}_R$ of polynomial derivations is a free  rank $n$ module:
\[ \mathrm{Der}_R\simeq\oplus_{1\le i\le n}R\partial_{i}.\]

\smallskip

To a square-free polynomial $f\in R$,  one associates its module of tangent derivations (or module of tangent vector fields): 
$$\mathrm{Der}(f)=\{\delta \in \mathrm{Der}_R\,|\, \delta(f)\in (f)\},$$
where $(f)$ is the principal ideal generated by $f$.
The module 
 $\mathrm{Der}(f)$ is a rank $n$ reflexive  $R$-module (see \cite[Cor. 1.7]{saito:logarithmic}) that could also be defined by the canonical exact sequence 
\[
0 \longrightarrow \mathrm{Der}(f) \longrightarrow \mathrm{Der}(R)\simeq R^n\longrightarrow R/(f) 
\]
where the last map is 
\[
\mathrm{Der}(R)\simeq R^n\longrightarrow R/(f), \,\, \delta \mapsto \delta(f) \,\mathrm{mod} \,(f).
\]
\begin{rmk}
 When $V(f)$ is a normal crossing divisor (NCD for short) then $\mathrm{Der}(f) $ is locally free, i.e. the localization $\mathrm{Der}(f)_P$ is free over $R_P$ for every prime ideal $P$ of $R$  (see \cite{deligne:equations} for instance). However, the converse is not true.    
\end{rmk}

\begin{rmk} \label{rem12}
    The map $\mathrm{Der}(R)\simeq R^n\longrightarrow R/(f)$ is the composition 
    of $R^n \longrightarrow R$ sending $\delta$ on $\delta (f)$ and the quotient map $R\longrightarrow R/(f)$ sending a polynomial 
    on its reduction modulo $f$. The image of the first map is the Jacobian ideal $J\subset R$ generated by the partial derivatives 
    $\partial_{i} f$. Its reduction modulo $(f)$ is $\frac{J+(f)}{(f)}$ ; when $f$ is homogeneous, $J+(f)=J$. It defines a subscheme in $V(f)$ supported by the singular locus of $V(f)$. This locus lives in codimension greater than $1$ since $f$ is square-free. So the exact sequence 
    \[
0 \longrightarrow \mathrm{Der}(f) \stackrel{\Phi}\longrightarrow  R^n\longrightarrow \frac{J+(f)}{(f)} \longrightarrow 0
\]
shows that $\bigwedge^n \Phi=\mathrm{det}(\Phi)=f$ and $\bigwedge^{n-1} \Phi=J$.

\end{rmk}

\begin{rmk} \label{rmk:1.3}
    If $f$ is homogeneous then
    $$\delta_E:=\sum_i x_i\partial_{i} \in \mathrm{Der}(f). $$
    The derivation $\delta_E$ is known as Euler derivation.
    This induces a splitting $$ \mathrm{Der}(f)=R\delta_E\oplus \mathrm{Der}_0(f)
   \,\, \mathrm{where}\,\, \mathrm{Der}_0(f)=\{\delta \in \mathrm{Der}(R)\,|\, \delta(f)=0\};$$
   to be precise, if $f$ is homogeneous and $\delta \in \mathrm{Der}(f) $ then $\delta-\frac{1}{\mathrm{deg}(f)}\frac{\delta(f)}{f}\delta_E \in \mathrm{Der}_0(f).$
   The module $\mathrm{Der}_0(f)$ is the Syzygy module of $\nabla f=(\partial_1f, \ldots, \partial_n f)$, 
   \[
0 \longrightarrow \mathrm{Der}_0(f) \longrightarrow R^n\stackrel{\nabla f} \longrightarrow R
\]
\end{rmk}
\begin{rmk}
    $\mathrm{Der}(f) $ being reflexive, it must be locally free on the projective plane $\p^2$  (resp. free on the affine plane $\mathbb{A}^2$).
\end{rmk}

\begin{rmk}
    When $V(f) $ is an hyperplane arrangement in $\p^n$, that is 
 $V(f)=\cup_iH_i$ is a union of $N+1$ distinct hyperplanes in $\p^n$ then 
 $\mathcal{T}_f=\widetilde{\mathrm{Der}(f)}$ is the dual of a Steiner sheaf, more precisely it arises as the kernel of a matrix of linear forms that could be described by the linear relations between the $H_i$ (see \cite{dolgachev-kapranov:arrangements}, \cite{faenzi-valles:london}, \cite{arrondo:schwarzenberger}, \cite{ancona-ottaviani:steiner},  etc.):

 \[
0 \longrightarrow \mathcal{T}_f \longrightarrow \sO(-1)^{N-1}\longrightarrow \sO^{N-n-1}.
\]

\end{rmk}


As first pointed out by Saito in \cite{saito:logarithmic}, there are divisors $V(f)$ not NCD such that $\mathrm{Der}(f)$ is a free $R$-module (meaning that  $\mathrm{Der}(f)\simeq R^n$). In the same paper, Saito gave a determinantal characterization of freeness, commonly called the Saito criterion.  This note aims to prove again this criterion (see Theorem \ref{saito-classic}) with basic algebraic arguments to generalize it to nonreduced divisor (see Theorem \ref{saito-multi} about the module $\mathrm{Der}_{\mathbf{m}}(f_1^{m_1}\cdots f_r^{m_r} )$ of multiderivations) and overall to finite families of polynomials (see Theorem \ref{pluri-poly}  where we define a rank $n-s+1$ module $\mathrm{Der}(f_1,\ldots, f_s)$ of logarithmic derivations associated to $s$ reduced polynomials). All these avatars of the Saito criterion are direct consequences of an elementary statement about reflexive modules which are explained in the next section \ref{Algebraic prelude}.

\medskip

Except when it is specified, the polynomials considered in this text are not necessarily homogeneous.

\section{Algebraic prelude: key proposition}
\label{Algebraic prelude}
We begin with an elementary proposition concerning reflexive modules. This proposition will be the key argument in the next proof of the Saito criterion and its generalization. 
\begin{prop}\label{basic-determinant}
  Let $M$ be a rank  $s\le n$ reflexive $R$-module and $\Phi$ a monomorphism $M\hookrightarrow R^n$.
  Let $N=\oplus_{1\le i\le s} Re_i$ be a free sub-module of $M$ and $\iota: N \hookrightarrow M$ be the inclusion map.
  Then 
  $$\bigwedge^s (e_1,\ldots,e_s)=\bigwedge^s (\Phi \circ \iota) =\mathrm{det}(\iota)\bigwedge^s \Phi.$$
  In particular, the following statements are equivalent:
  \begin{enumerate}
      \item $M=N$ 
      \item $\det(\iota)\in k^*$ 
      \item $V(\bigwedge^s (\Phi \circ \iota))=V(\bigwedge^s \Phi)$
  \end{enumerate}
\end{prop}
\begin{proof}
This is a consequence of the commutative square:
   \begin{equation}
\begin{split} \xymatrix@-0ex{ 
&0 \ar[d] & & &\\
& N\ar@{=}[r]\ar[d]^{\iota} & N  \ar[d]^{(e_1,\ldots, e_s)} \\
0\ar[r] & M \ar[r]^{\Phi} & R^n 
} \end{split}
\end{equation} 
Indeed, $N$ and $M$ being reflexive modules of rank $s$, 
$\bigwedge^s N\simeq \bigwedge^s N^{\vee \vee}$ and
$\bigwedge^s M\simeq \bigwedge^s M^{\vee \vee}$ are reflexive modules of rank one that are free modules and 
there is a commutative square
 \begin{equation}
\begin{split} \xymatrix@-0ex{ 
&0 \ar[d] & & &\\
& \bigwedge^s N\ar@{=}[r]\ar[d]_{\det(\iota)} & \bigwedge^s N  \ar[d]^{\bigwedge^s (e_1,\ldots, e_s)} \\
0\ar[r] & \bigwedge^s M \ar[r]^-{\bigwedge^s \Phi} & \bigwedge^s R^n 
} \end{split}
\end{equation} 
showing that $$\bigwedge^s (e_1,\ldots, e_s)=\bigwedge^s (\Phi \circ \iota)= \mathrm{det}(\iota)\bigwedge^s \Phi.$$This proves our proposition.
Let us point out that if $s=n$ then $\det(\Phi)\;|\, \det(e_1\ldots, e_s).$
\end{proof}

\section{Free divisors: Saito criterion}
\begin{dfn}
  The module $\mathrm{Der}(f)$ is free  of rank $n$ if and only if there exist $n$ derivations
$$\delta_i=\sum_{1\le j\le n}P_{ij}\partial_{j} \in \mathrm{Der}(f)$$ such that 
$$ \mathrm{Der}(f)=\oplus_i R\delta_i.$$  
\end{dfn}

Let us first point out the following crucial fact (see \cite[1.5]{saito:logarithmic}):
\begin{lem}\label{pre-rem}
Let
$\delta_i\in \mathrm{Der}(f)$ be  $n$ independent derivations tangent to $V(f)$. Then, 
  $$f\, | \,\mathrm{det}(\delta_1,\ldots, \delta_n):=\mathrm{det}(P_{i,j})_{1\le i,j\le n}.$$ 
\end{lem}
\begin{proof}
Since $\delta_i$ are  $n$ independent derivations tangent to $V(f)$ the $R$-module  $\oplus R\delta_i$ is a free sub-module of rank $n$ of 
$\mathrm{Der}(f)$. Calling $N=\oplus R\delta_i$, $M=\mathrm{Der}(f)$ and $\Phi$ the inclusion $\mathrm{Der}(f)\hookrightarrow \mathrm{Der}(R)\simeq R^n$ we have, according to Remark \ref{rem12}, $\det(\Phi)=f$ and according to Proposition \ref{basic-determinant},
\[ \det(\Phi)=f\;|\, \mathrm{det}(\delta_1,\ldots, \delta_n)
\]
\end{proof}
As a direct consequence, we get the classical Saito criterion (see \cite[Theorem (ii), page 270]{saito:logarithmic}):
\begin{thm}
\label{saito-classic}
$ \mathrm{Der}(f)=\oplus_i R\delta_i$ if and only if 
$\mathrm{det}(\delta_1,\ldots, \delta_n)=uf, \,\, u\in k^*.$    
\end{thm}
\subsection{First proof of Theorem \ref{saito-classic}}
\begin{proof}
   Denoting $N=\oplus R\delta_i$ and $M=\mathrm{Der}(f)$ we have, according to Proposition \ref{basic-determinant}, $M=N$ if and only if
   $\det(\iota)=u\in k^*.$
\end{proof}

\subsection{Second proof of Theorem \ref{saito-classic}}
\begin{proof}

Let  $U=\oplus_{j=1,\ldots, n}R\delta_j$ be a free sub-module
of $\mathrm{Der}(f)$. We get a commutative diagram:
\begin{equation}
\begin{split} \xymatrix@-2ex{ 
&0 \ar[d] & & &\\
0\ar[r] & U\ar[r]\ar[d] & R^n\ar[r]\ar@{=}[d] & K \ar[d]\ar[r] & 0 \\
0\ar[r] & \mathrm{Der}(f)\ar[r]\ar[d] & R^n \ar[r] & \frac{J+(f)}{(f)} \ar[r]\ar[d] & 0 \\
& \mathrm{Der}(f)/U \ar[d]  &  & 0& \\
& 0 &  & &
} \end{split}
\end{equation}
where the singular locus of the cokernel $K$ is defined by $V(\bigwedge^{n} U)=V(\mathrm{det}(\delta_i)).$

\smallskip

Since $ \mathrm{Der}(f)$ is reflexive, $U$ is free and 
$\mathrm{rk}(\mathrm{Der}(f))=\mathrm{rk}(U)$,  then $\mathrm{Der}(f)/U=0 $ when
$\mathrm{Der}(f)=U$ 
or  $\mathrm{Der}(f)/U $ is a torsion sheaf when $\mathrm{Der}(f)\neq U .$   

\smallskip

--- If   $U=\mathrm{Der}(f)$ then it is is free and 
$K=\frac{J+(f)}{(f)}$ by the snake lemma. This proves that the singular locus $V(\bigwedge^{n} U)$ of $K$ is exactly $V(f)$. Actually, this proves more. It proves that the Jacobian ideal $J$  coincides with the first Fitting ideal defined by $V(\bigwedge ^{n-1}U)$.

\smallskip

--- If    $\mathrm{Der}(f)\neq U$ then $ \mathrm{Der}(f)/U$ is a torsion sheaf. By the snake lemma, we get
\[
0 \longrightarrow \mathrm{Der}(f)/U \longrightarrow K \longrightarrow \frac{J+(f)}{(f)} \longrightarrow 0
\]
proving that the singular locus of $K$ defined by $V(\mathrm{det}(\delta_i))$ differs from $V(f)$.

\end{proof}

\section{Examples of applications}
We give here two kinds of applications, some in the homogeneous (or projective) case and some in the non-homogeneous (or affine) case. The first case is, in some sense, easier since the degree of the derivations are or are not the good ones, separating the free from the non-free divisors. In the affine case, the criterion appears to be stronger since the degrees are not enough to distinguish which set of derivations is a basis or not.

\begin{eg}
 The arrangement of $3n+3$ lines (the reflection arrangement denoted by $\mathcal{A}_3^3(n)$ in 
 \cite{orlik-terao:arrangements}) on $\p^2$
 defined by the equation  $$f(x,y,z):=xyz(x^n-y^n)(x^n-z^n)(y^n-z^n)=0$$
 is free with exponents $(n+1,2n+1)$. Indeed the derivations $$\delta_1=x^{n+1}\partial_x+y^{n+1}\partial_y+z^{n+1}\partial_z \;\mathrm{and}\; \delta_2=x^{2n+1}\partial_x+y^{2n+1}\partial_y+z^{2n+1}\partial_z$$
 belong to $\mathrm{Der}(f)$ and

\[
\mathrm{det}\left[
    \begin{array}{ccc}
    x&x^{n+1}&x^{2n+1}\\
    y&y^{n+1} & y^{2n+1}\\
    z&z^{n+1}& z^{2n+1}
     \end{array}
       \right]=xyz\mathrm{det}\left[
    \begin{array}{ccc}
    1&x^{n}&x^{2n}\\
    1&y^{n} & y^{2n}\\
    1&z^{n}& z^{2n}
     \end{array}
       \right]=xyz(x^n-y^n)(x^n-z^n)(y^n-z^n).
\]
\end{eg}

Let us give two examples now of smooth conics in the affine plane having a basis consisting of polynomials of degree 2 for the first one and of degree one for the other.       
\begin{eg}
    This exemple is due to Yohan Genzmer. 
    He shows that even if $\mathrm{Der}(f)$ contains a derivation of minimal degree this derivation is not necessarily a generator of  $\mathrm{Der}(f)$. Let $f(x,y)=x^2+y^2+x.$ Then the derivation $2y\partial_x-(2x+1)\partial_y \in \mathrm{Der}(f)$ in particular it belongs to $\mathrm{Der}_0(f)$ but a basis of $\mathrm{Der}(f)$ is given by 
\[
\left\{
    \begin{array}{ccc}
    \delta_1&=&2y(2x+1)\partial_x+(4y^2-1)\partial_y\\
    \delta_2&=&x(x+1)\partial_x+(\frac{y}{2}+yx)\partial_y
     \end{array}
       \right.
\]
 Indeed, note that $\delta_1(f)=8yf$, $\delta_2(f)=(2x+1)f$ so that $\mathrm{det}(\delta_1,\delta_2)=f.$
In addition, we observe that the curve $V(f)\subset \mathbb{A}_2$ is free, but its homogenization $V(x^2+y^2+xz)\subset \p^2$, which is a smooth conic, is not free. 
\end{eg}

\begin{eg}
The curve $C\cup L=\{z(x^2+yz)=0\}$ union of a smooth conic and one tangent line in $\p^2$ is free with exponents $(1,1)$. 
    Indeed two derivations in $\mathrm{Der}_0(C\cup L)$  are 
    $z\partial_x-2x\partial_y$ and $x\partial_x+4y\partial_y-2z\partial_z.$ and 
    $$\mathrm{det} 
\begin{pmatrix}  x & z  &x  \\
y & -2x  &4y \\
  z & 0&-2z
   \end{pmatrix}= 6z(x^2+yz).$$
    Removing the line $L=\{z=0\}$ we get

  \[
0 \longrightarrow \mathcal{T}_{C\cup L}=\sO_{\p^2}(-1)^2\longrightarrow \mathcal{T}_{C} \longrightarrow \sO_L(\alpha) \longrightarrow 0,
\]
for some integer $\alpha$.
Then on the affine space $\p^2\setminus L$, the curve becomes $\{y+x^2=0\} $ which is free with two derivations of degree 1: 
\[
\left\{
    \begin{array}{ccc}
    \delta_1&=&\partial_x-2x\partial_y\\
    \delta_2&=&\frac{x}{2}\partial_x+y\partial_y
     \end{array}
       \right.
\]
verifying $\det(\delta_1,\delta_2)=y+x^2.$
\end{eg}

\section{Multiderivations}

The previous material can be extended to non-reduced divisors and their associated modules of multiderivations. Indeed a multiple polynomial $f=f_1^{m_1}\cdots f_r^{m_r}$, where $f_i$ are square-free and irreducible polynomials of degree $r_i\ge 1$,
being given one can define the following sub-module of $\mathrm{Der}(R)$ (see \cite{ziegler:multi}):
$$ \mathrm{Der}_{\mathbf{m}}(f):=\{\delta \in \mathrm{Der}_R\,|\, \delta(f_i)\in (f_i^{m_i}) \,\, 1\le i\le r  \}.$$
This $R$ module which is also reflexive, can be defined  by the canonical exact sequence 
\[
0 \longrightarrow \mathrm{Der}_{\mathbf{m}}(f) \longrightarrow \mathrm{Der}(R)\simeq R^n\longrightarrow \bigoplus_{i=1}^r R/(f_i^{m_i}) 
\]
where the last map is given by 
\[
\delta \mapsto \Big( \delta(f_1) \,\mathrm{mod} \,(f_1^{m_1})\, , \dots, \,\delta(f_r) \,\mathrm{mod} \,(f_r^{m_r}) \Big) .
\] 
In particular, this shows that $\rk(\mathrm{Der}_{\mathbf{m}}(f))=n$. 
Moreover, there is a natural inclusion 
$$\mathrm{Der}_{\mathbf{m}}(f)\subset \mathrm{Der}(\prod_i f_i).$$
 Let us denote by $\Phi$ the first map in the exact sequence above, that is the inclusion map: $\mathrm{Der}_{\mathbf{m}}(f) \longrightarrow \mathrm{Der}(R)\simeq R^n.$
 Like in the reduced case we have an exact sequence 
\[
0 \longrightarrow \mathrm{Der}_{\mathbf{m}}(f) \stackrel{\Phi} \longrightarrow  R^n\longrightarrow \mathrm{coker}(\Phi) \longrightarrow 0,
\]
where the support of the module $\mathrm{coker}(\Phi)$ is $V(\mathrm{det}(\Phi))$.\\
Let us verify that  $\mathrm{det}(\Phi)=f=f_1^{m_1}\cdots f_r^{m_r}$. Since both modules have the same rank then $V(\mathrm{det}(\Phi))$ is empty and $\Phi$ is an isomorphism or $V(\mathrm{det}(\Phi))$ defines a divisor. Of course a general derivation  does  not belong to $\mathrm{Der}_{\mathbf{m}}(f)$ which proves that  $V(\mathrm{det}(\Phi))\neq \emptyset.$ Multiplying $f$ by $f_i$ we obtain a natural exact sequence: 
\[
0 \longrightarrow \mathrm{Der}_{\mathbf{m'}}(f) \longrightarrow  \mathrm{Der}_{\mathbf{m}}(f)  \longrightarrow R/(f_i),
\]
where $m'=(m_1,\ldots, m_i+1,\ldots, m_n)$, the first map is the inclusion and the second map sends a multiderivation 
$\delta$ to $\frac{\delta(f_i)}{f_i^{m_i}} \mod \, (f_i).$ The commutative square 
$$\begin{CD}
@.  0 \\
@. @VVV  \\
 0@>>>  \mathrm{Der}_{\mathbf{m'}}(f) @>{\Phi'}>> R^n \\
@.  @VV{\iota}V @| \\
0@>>>   \mathrm{Der}_{\mathbf{m}}(f) @>{\Phi}>> R^n 
\end{CD}$$
 gives $\mathrm{det}(\Phi')=\mathrm{det}(\iota)\mathrm{det}(\Phi)=f_i\mathrm{det}(\Phi)$ proving the statement by induction.
Saito's criterion is still pertinent for this module, as this is proved in \cite[Theorem and Definition 8]{ziegler:multi}:
\begin{thm}
\label{saito-multi}
$  \mathrm{Der}_{\mathbf{m}}(f) =\oplus_i R\delta_i$ if and only if 
$\mathrm{det}[\delta_1,\ldots, \delta_n]=uf, \,\, u\in k^*.$    
\end{thm}
\begin{proof}
Denoting $N=\oplus R\delta_i$ and $M=\mathrm{Der}_{\mathbf{m}}(f) $ we have, according to Proposition \ref{basic-determinant}, $M=N$ if and only if
   $\det(\iota)=u\in k^*.$
\end{proof}


\section{Saito criterion for several polynomials}
Let $f_1, \ldots,f_s \in R$ be $1<s\le n-1$ algebraically independent polynomials and  $S=V(f_1, \ldots,f_s)$. We denote by $I_S=(f_1,\ldots, f_s)$ the ideal defining $S$. The module of logarithmic derivations associated to {$I_S$, say  $\mathrm{Der}(I_S)$}, consists in derivations $\delta$  such that 
$\delta (g)\in I_S$ for all $g\in I_S$ {(see for instance \cite{Miranda-Neto})}. This is a rank $n$ module; one can verify easily that it contains $\mathrm{Der}(\prod_i f_i).$

\smallskip

In this section, we are interested in the module:
\[
\mathrm{Der}(f_1,\ldots, f_s)=\{\delta \in \mathrm{Der}(R)\simeq R^n \,|\, 
f_j\delta(f_i)=f_i\delta(f_j), \,\, 1\le i<j\le s\},
\]
 and its submodule
\[
\mathrm{Der}_0(f_1,\ldots, f_s)=\{\delta \in \mathrm{Der}(R)\simeq R^n \,|\, 
\delta(f_i)=0, \,\, 1\le i\le s\},
\]
whose associated sheaf was studied in \cite{faenzi-jardim-valles}; it can be regarded as a generalization of the $\mathrm{Der}_0$ module for a single polynomial, which was mentioned in Remark \ref{rmk:1.3}.
We remark first that this module depends on the choice of the polynomials generating $I_S$.
{Actually, $\mathrm{Der}(f_1,\ldots, f_s)$ is a submodule of $\mathrm{Der}(\prod_i f_i)$, strictly contained in it when $s\ge 2$ for reason of ranks as we will see in Theorem \ref{pluri-poly} below. When $\mathrm{deg}(f_1)\neq \mathrm{deg}(f_2)$ one can also verifies easily that $\delta_E \in  \mathrm{Der}(f_1f_2)$ but $\delta_E \notin \mathrm{Der}(f_1, f_2)$.}


\begin{thm}
\label{pluri-poly}
Let $f_1, \ldots,f_s \in R$ be $1<s\le n-1$ algebraically independent polynomials. Then, 
\begin{enumerate}
    \item $\mathrm{Der}(f_1,\ldots, f_s)$  is a reflexive module of rank $n-s+1.$
    \item Let $\delta_1, \ldots, \delta_{n+1-s}$ be independent derivations in $\mathrm{Der}(f_1,\ldots, f_s)$.
Then,
     \[
\mathrm{Der}(f_1,\ldots, f_s)=\bigoplus_{j=1,\ldots, n+1-s}R\delta_j \Leftrightarrow
\mathrm{codim}_{k^n} V\left( \bigwedge^{n+1-s}  [\delta_1, \ldots, \delta_{n+1-s}]\right)
       >1.
       \]
\end{enumerate}
\end{thm}
\begin{proof}

Let us consider the $R$-module map: 
$$ \mathrm{Der}(R)\simeq R^n \longrightarrow R^{s-1}, \,\, \mu \mapsto (f_i\mu (f_1)-f_1\mu (f_i))_{2\le i\le s}.$$
The image of this map, let us say $K$, being a sub-module of a free module, is torsion-free. Since the $f_i$ are algebraically independent this map is of maximal rank. Being a map between two free modules, its kernel is reflexive.
Moreover, the kernel of this map is the module
$\mathrm{Der}(f_1,\ldots , f_s)$ since it coincides with the kernel of the map 
$$ \mathrm{Der}(R)\simeq R^n \longrightarrow R^{\binom{s}{2}}, \,\, \mu \mapsto (f_i\mu (f_j)-f_j\mu (f_i))_{1\le i<j\le s}.$$ 
   This proves (1).

\medskip

Considering now the free sub-module $U=\oplus_{j=1,\cdots, n+1-s}R\delta_j$ 
of $\mathrm{Der}(f_1,\ldots , f_s)$ we obtain a  commutative diagram:
\begin{equation}
\begin{split} \xymatrix@-2ex{ 
&0 \ar[d] & & &\\
0\ar[r] & U\ar[rr]^{\Psi}\ar[d]^{\iota}&& R^n\ar[r]\ar@{=}[d] & K^{'} \ar[d]\ar[r] & 0 \\
0\ar[r] & \mathrm{Der}(f_1,\ldots , f_s)\ar[rr]^{\Phi}\ar[d] && R^n \ar[r] & K \ar[r] & 0 \\
& \mathrm{Der}(f_1,\ldots ,f_s)/U \ar[d]  &  & & \\
& 0 &  & &
} \end{split}
\end{equation}
 Since  $K$ is torsion-free this implies that $$\mathrm{codim}_{k^n} V\left(\bigwedge^{n+1-s} \Phi\right)>1.$$ Moreover, by Proposition \ref{basic-determinant},  $\bigwedge^{n-s+1} \Psi=\mathrm{det}(\iota)\bigwedge^{n-s+1} \Phi$.  So the alternative is: 
 \begin{itemize}
     \item $\mathrm{det}(\iota)\in k^*$ and $\mathrm{codim}_{k^n} V\left(\bigwedge^{n+1-s} \Psi\right) =\mathrm{codim}_{k^n} V\left(\bigwedge^{n+1-s} \Phi\right)>1$;
     \item $\mathrm{deg}(\mathrm{det}(\iota))>0$ and $\mathrm{codim}_{k^n} V\left(\bigwedge^{n+1-s} \Psi\right)=\mathrm{codim}_{k^n} V\left(\mathrm{det}(\iota))\right)=1.$
 \end{itemize} This proves  
 that 
 \[
 U=\mathrm{Der}(f_1,\ldots , f_s) \Leftrightarrow \mathrm{det}(\iota)\in k^* \Leftrightarrow \mathrm{codim}_{k^n} V\left(\bigwedge^{n+1-s} \Psi\right)>1.
 \]
 \end{proof}

\begin{eg}
    Let us give an example in $\mathbb{A}_3$. Let $f(x,y,z)=x^2-y^2$ and $g(x,y,z)=z^3+1.$
    The image of the map $$R^3 \longrightarrow R ~~, ~~ \mu \mapsto f\mu(g)-g\mu(f)$$
    is the ideal $(x(z^3+1), y(z^3+1), z^2(x^2-y^2))$.
    The derivations 
\[
\left\{
    \begin{array}{ccc}
    \delta_1&=&y\partial_x+x\partial_y\\
    \delta_2&=&xz^2\partial_x+yz^2\partial_y-(z^3-1)\partial_z
     \end{array}
       \right.
\]
belong to $\mathrm{Der}(f,g)$ (actually $\delta_i(f)=\delta_i(g)=0$). Since $\mathrm{codim}V(\bigwedge^2(\delta_1,\delta_2))>1$,
the module $\mathrm{Der}(f,g)$ is free, more precisely, 
$$ \mathrm{Der}_0(f,g)=\mathrm{Der}(f,g)\simeq R\delta_1\oplus R\delta_2. $$ 
\end{eg}

\begin{eg}
The previous example comes from the simple example in $\p^3$ given by the two equations:
$$ f=x_0^2-x_1^2, \; g=x_2^3+x_3^3.$$
 The rank 3 module $\mathrm{Der}(f,g)$ is the kernel of the graded map $R^4\rightarrow R[4]$ sending a derivation $\mu$ to the polynomial $f\mu(g)-g\mu(f).$ 
The derivations 
\[
\left\{
    \begin{array}{ccc}
    \delta_1&=&x_1\partial_{x_0}+x_0\partial_{x_1}\\
    \delta_2&=&x_3^2\partial_{x_2}-x_2^2\partial_{x_3}\\
    \delta_3&=& 3x_0\partial_{x_0}+3x_1\partial_{x_1}+2x_2\partial_{x_2}+2x_3\partial_{x_3}
     \end{array}
       \right.
\]
belong to $\mathrm{Der}(f,g)$. Since $\mathrm{codim}V(\bigwedge^3(\delta_1,\delta_2, \delta_3))>1$,
the module $\mathrm{Der}(f,g)$ is free, more precisely, 
$$\mathrm{Der}(f,g) = R\delta_1\oplus R\delta_2 \oplus R\delta_3=R[-1]\oplus R[-2]\oplus R[-1].$$
Observe, in addition, that $\delta_i(f)=\delta_i(g)=0$, for $i=1,2$; this means that
$$ \mathrm{Der}_0(f,g) = R\delta_1\oplus R\delta_2 = R[-1]\oplus R[-2]. $$
\end{eg}

\bibliographystyle{amsalpha}
\bibliography{basic_algebraic_geometry}

\end{document}